\documentclass[12pt,a4paper]{article}
\usepackage[utf8]{inputenc}\usepackage[novbox]{pdfsync}\usepackage{float}\usepackage{array}
\usepackage[T1]{fontenc}\usepackage{tikz,qtree}
\usepackage{amsmath,amsfonts,amssymb,amsthm,bbm,graphicx,xcolor,latexsym,subfig,
placeins,etex,caption,keyval,url}
\usepackage{multirow}
\usepackage{bbm,amsfonts}\usepackage{pdfsync}
 \def\ssec{\subsection} \def\Kol{Kolmogorov }
\def\equ{equation }     \def\rv{random variable } 
   \def\Tc{\tilde{c}} \def\Tq{\tilde{q}} \def\Tl{{\tilde{\lambda}}}  \def\mW{{\mathcal W}} \def\w{{\mathbf w}}  \def\fun{function } \def\funs{functions } \def\rui{\psi}
\def\GS{Gerber-Shiu }   \newcommand{\bff}[1]{{\mbox{\boldmath$#1$}}}
\def\ts{two-sided } \def\wk{well-known } \def\pros{probabilities }
   \def\upm{up to a multiplicative constant}
\def\abs{absolute }  \def\rps{ruin probabilities } \def\ren{renewal equation }
 \def\wk{well-known} \def\ts{two sided exit }

\def\G{\Gamma}
  
\newcommand{\kil}{\mathbf{e}} \def\procs{processes} \def\proc{process}
\def\Rui{\Psi} \def\sRui{\ovl{\Rui}}
 \def\fp{first passage }
\def\ovl{\overline}
\def\lev{L\'evy }  \def\mL{{\mathcal L}} \def\rp{ruin probability } 
 \def\sn{spectrally negative }
\long\def\symbolfootnote[#1]#2{
\begingroup
\def\thefootnote{\fnsymbol{footnote}}\footnote[#1]{#2}
\endgroup}
\def\fn{\symbolfootnote}
\def\I{\infty} \def\Eq{\Leftrightarrow}
  \def\T{\widetilde}
\def\CL{Cram\`er-Lundberg }  
 \def\PH{phase-type }  
\def\BEN{\begin{enumerate}}  \def\BI{\begin{itemize}}
\def\EEN{\end{enumerate}}   \def\EI{\end{itemize}} \def\im{\item} \def\Lra{\Longrightarrow}  \def\eqr{\eqref}
\def\no{\nonumber} 
\def\mG{{\mathcal G}}
\def\g{\gamma}     \def\b{\beta}
  \def\th{\theta} 
 
\def\k{\kappa} \def\l{\lambda} \def\a{\alpha} 
\def\Lm{L\'evy measure }  
\def\lm{\Pi}  \def\ith{it holds that }
  \def\r{r} \def\s{\sigma}   
  \def\bc{\begin{cases}
  }    
\def\ec{\end{cases}} \def\expo{exponential } \def\CIR{Cox-Ingersoll-Ross }
  \def\qu{\quad} \def\for{\forall}
  \def\beq{\begin{eqnarray}} \def\eeq{\end{eqnarray}}
   \def\be{\begin{equation}} \def\ee{\end{equation}}
\def\bea{\begin{eqnarray*}} \def\prf{{\bf Proof:} }
\def\eea{\end{eqnarray*}} \def\la{\label}
\def\LE{Laplace exponent } \def\LT{Laplace transform} \def\BM{Brownian motion }
\def\LTs{Laplace transforms } \def\q{q} 
\def\le{\left} \def\ri{\right}
\def\nne{nonnegative } \def\fr{\frac}    \def\ta{T_{a,-}}  \def\tb{T_{b,+}} \def\tl{T_{l,-}}
 \def\sec{\section} \def \fe{for example } \def\wlo{w.l.o.g. }  \def\ith{it holds that }

 \def\deF{de Finetti } 

\newtheorem{Thm}{Theorem}
\newtheorem{Ass}{Assumption}
\newtheorem{Lem}{Lemma}
\newtheorem{Exa}{Example}

\newtheorem{Cor}{Corollary}
\newtheorem{Def}{Definition}
\newtheorem{Rem}{Remark}
\newtheorem{Exe}{Exercice}

\def\beXe{\begin{Exe}} \def\eeXe{\end{Exe}}
\def\eeD{\end{Def}} \def\beD{\begin{Def}}
\def\beXa{\begin{Exa}} \def\eeXa{\end{Exa}}
\def\beR{\begin{Rem}} \def\eeR{\end{Rem}}
\def\beL{\begin{Lem}} \def\eeL{\end{Lem}}
\newtheorem{Pro}[Lem]{Proposition}
\def\beP{\begin{Pro}} \def\eeP{\end{Pro}}
\def\beC{\begin{Cor}} 
\def\beT{\begin{Thm}}
  \def\eeT{\end{Thm}}   \def\Mar{Markov }
\def\eeC{\end{Cor}}
\newtheorem{Qu}{Problem}
\def\beQ{\begin{Qu}} \def\eeQ{\end{Qu}}
  \def\sur{survival }
 
\def\pros{probabilities } 
\def\rvs{random variables } \def\rv{random variable } 
  \def\expo{exponential } \def\pro{probability }
\def\prob{problem } \def\probs{problems}  \def\ts{two-sided exit }
  \def\H{\widehat }
   \def\Ga{\Gamma}
\def\1{\mathop{\rm 1\!\!I}\nolimits}   
 \def\mZ{{\mathcal Z}}  \def\mW{{\mathcal W}}
\newcommand{\pd}[2]{\frac{\partial #1}{\partial #2}}
 \def\Seg{Segerdahl}
\def\prob{problem} \def\probs{problems} \def\proba{probability} \def\str{strong Markov property}\def\expoj{exponential jumps}

 \def\OU{Ornstein-Uhlenbeck } 
 \def\sp{spectrally positive } 
 \newcommand{\md}{\mathrm d}   \def\divs{dividends }

\newcommand{\blue}{\textcolor[rgb]{0.00,0.00,1.00}}

     \def\sd{scale derivative }
     
      \def\woF{\widehat{\overline{\lm}}}
     \def\Xt{(X_t)_{t \geq 0}} \def\mZ{{\mathcal Z}}
    \def\Lm{L\'evy measure }\def\wr{with respect to }
   \def\Mp{More precisely, } \def\sats{satisfies}  \def\fun{function }  \def\funs{functions}    \def\resp{respectively} \def\proc{process} \def\eq{equation} \def\eqs{equations} \def\cd{(\cdot)} \def\fac{factorization}\def\fno{from now on} \def\upc{up to a constant}\def\std{state dependent }\def\satg{satisfying }\def\c{c}\def\SL{Sturm Liouville }\def\Fq{\Phi_\q}\def\prc{proportionality constant}

\providecommand{\abs}[1]{\left\lvert#1\right\rvert}
\providecommand{\pr}[1]{\left(#1\right)} 
\providecommand{\pp}[1]{\left[#1\right]} 

\newcommand{\figu}[3]{
\begin{figure}[!h]
\begin{center}
{\includegraphics[width=13 cm, height=8 cm]{#1}}
\end{center}

\vspace{-0.2cm}
\caption{\hspace{0.25cm}#2\label{f:#1}}
\end{figure}
}

\begin{document}
\title{A Review of First-Passage Theory for the Segerdahl Risk Process and Extensions}

\newcommand{\orcidauthorA}{0000-0002-2615-0950} 

\author{{Florin Avram} 
 and {Jose-Luis Perez-Garmendia}
 }


\maketitle

\abstract{The Segerdahl process (Segerdahl (1955)), {characterized by}
exponential claims and
affine drift, has drawn a~considerable amount of interest---see, for example, (Tichy (1984); Avram and Usabel (2008); Albrecher et al. (2013); Marciniak and Palmowski (2016)), due to its
economic interest (it is the simplest risk process which takes into account the effect of interest
rates). See  (Albrecher and Asmussen 2010, Chapter 8) for an excellent overview, including extensions to processes with state dependent drift. It is also the simplest
non-Lévy, non-diffusion example of a spectrally negative Markov risk model. Note that for both
spectrally negative Lévy and diffusion processes, first passage theories which are based on identifying
two “basic” monotone harmonic functions/martingales have been developed. This means that for
these processes many control problems involving dividends, capital injections, etc., may be solved
explicitly once the two basic functions have been obtained. Furthermore, extensions to general
spectrally negative Markov processes are possible (Landriault et al. (2017), Avram et al. (2018);
Avram and Goreac (2019); Avram et al. (2019b)). Unfortunately, methods for computing the basic
functions are still lacking outside the Lévy and diffusion classes, with the notable exception  of the Segerdahl process, for which the
ruin probability has been computed (Paulsen and Gjessing (1997). As a consequence,
the $W$  scale function may be computed as well, via simple \proba \ arguments which apply a priori to all \procs\ with \expoj\ (and may be extended to phase-type jumps as well). Further work going beyond \expoj\ and linear drifts has been provided in provided in (Avram and Usabel (2008)) and  (Czarna et al. (2017)), respectively.
However, there is a
striking lack of numerical results in both cases. This motivated us  to review  these approaches, with the
purpose of drawing attention to connections between them, and underlining open problems.}


\sec{Introduction and Brief Review of First Passage Theory for Spectrally Negative
\Mar \ Processes}

 To set the stage for our topic and future research,
consider a~\sn
jump diffusion on a~filtered probability space $(\Omega,  \{\mathcal{F}_t\}_{t \geq 0}, P)$, which~ satisfies the~SDE:
\begin{eqnarray}  \label{jd}
d X_t&= & c(X_t) d t + \sigma(X_t) d B_t  - d J_t, \;  J_t=\sum_{i=1}^{N_\l(t)} C_i, \; \for X_t > 0
\end{eqnarray}
and is  {absorbed} 
when leaving a half line $(l,\I)$.\fn[4]{The boundary point $l$ may be a natural barrier, like the largest root of $\s(x)$. Or, when $\s(x)=0$ and $c(x)$ is increasing, $l$  may be the largest root of $c(x)$, called absolute ruin point.
 Or, it can be a point below which the process is artificially killed.}
 Here, $C_i$ are \nne \; i.i.d. \rvs with distribution measure $ F_C(dz)$ and finite mean, $B_t$ is an independent standard Brownian motion, $\s(x) >0,  c(x)>0, \for x > l$, $N_\l(t)$ is an independent Poisson process of intensity $\l$. {The functions $c(x), \; v(x):=\fr{\s^2(x)}2$ and $\lm(d z)=\l F_C(dz)$ are referred to as the
  \lev-Khinchine
 characteristics of $X_t$.}

Note that we assume that all jumps go in the same direction and have constant intensity so that we can take advantage of potential simplifications of the first passage theory in this case.  See  \cite{paulsen2010ruin} and  \cite[Chapter 8]{AA} for further information on risk processes with state dependent drift, and~in particular the two pages of historical notes and references in the last reference.

{ {The Segerdahl process}} is the simplest example outside the spectrally negative \lev and diffusion classes. It is obtained by assuming $v(x)=0$ in \eqr{jd}, and~$C_k$ to be exponential i.i.d random variables with density $f(x)=\mu e^{-\mu x}$--see \cite{Seg} for the case $c(x)=c + rx,  r>0, c \geq 0$, and see also~\cite{Tichy84} for nonlinear $c(x)$.
The \Seg\ \ \proc \ \sats\ thus the SDE
\bea d X_t= \pr{\c + \r X_t} dt -d \pr{\sum_{i=1}^{N_t}C_i}, \qu  \r>0, \c \geq 0. \eea
It is  given explicitly by
\begin{align*}
X_t=X_0+\r \int_0^{t}X_s ds+K_t, \quad \Eq \quad
X_t=X_0e^{ \r  t} +e^{ \r  t} \int_0^t e^{-\r  s} d K_s,
\end{align*}
with $K_t=\c t  -\sum_{i=1}^{N_t}C_i$ being a Cram\'er-Lundberg process, whose Laplace exponent is
\[
\k(\theta):=\frac{1}{t}\log \mathbb{E}[e^{\th K_t}]=\c \th + \lambda(\frac{\mu }{\mu +\th}-1).
\]

\beR An essential point for the Segerdahl process is the fact that the point  $-\fr{\c}{\r}
$  is an \abs ruin~level, in the sense that after a jump below this point, the process will never cross back. We may assume \wlo that the \abs ruin~level is $0$, or, equivalently, that $\c=0$. \eeR

{{First passage theory}} concerns the first passage times above and below fixed levels. For any process $\Xt$, these are defined by
\begin{equation}\la{fpt}
\begin{aligned}
\tb &= \tb^{X}=\inf\{t\geq 0: X_t> b\},\\
T_{a,-}&=T_{a,-}^X=\inf\{t\geq 0: X_t< a\},
\end{aligned}
\end{equation}
with $\inf\emptyset=+\infty$, and~the upper script $X$ typically omitted. Since $a$ is typically fixed below,
we will often write for simplicity $T $ instead of $\ta$.

 First passage times are important in the control of reserves/risk
 processes. The~rough idea is that when below low levels $a$, reserves processes should be replenished at some cost, and~when above high levels $b$, they should be partly invested to yield income---see, for example, the comprehensive textbook \cite{AA}.

 The most important first passage functions are the  two-sided upward and downward exit functions from a~bounded interval $[a,b]$, defined respectively by

\beq \la{Rui}
\bc \sRui^{b}_{q}(x,a)  &  :=E_{x}\left[  e^{-q\tb}{\mathbf 1}_{\left\{
\tb<\ta\right\}  }\right]=P_{x}\left[ \tb<\min(\ta, \kil_q)\right] \\
\Rui_{q}^{b}(x,a)  &  :=E_{x}\left[  e^{-q\ta}{\mathbf 1}_{\left\{  \ta<\tb\right\}  }\right]=P_{x}\left[ \ta<\min(\tb, \kil_q)\right]\ec  \; \;  q\geq0, a\leq x\leq b,
\eeq
where $\kil_q$ is an~independent \expo \rv of rate $q$. We will call them (killed) {{survival}} and {{ruin}} probabilities, respectively\footnote{See \cite{Ivakil} for a~nice exposition of killing.}, but the qualifier killed will be usually dropped below. The~absence of killing will be indicated by omitting the subindex $q$.
Note that in the context of potential theory, \eqr{Rui} are called equilibrium potentials \cite{blumenthal2007markov} (of the capacitors $\{b,a\}$ and~$\{a,b\}$).

{{\bf Beyond \rps: scale functions, dividends, capital gains, etc.} Recall that for ``completely asymmetric \lev" \procs, with jumps going all in the same direction, a~large variety of \fp {\probs \ may} be reduced to the computation of the two monotone ``scale functions'' $W_\q,Z_\q$---see, \fe, \cite{Suprun,Ber97,Ber,AKP,APP,APP15,APY,IP,AIZ,LZ17,LP,AZ}, and see \cite{AGV} for a~recent compilation of more than 20 laws expressed in terms of $W_\q,Z_q$.}

For example, for {spectrally negative L\'evy processes},
the {killed survival} probability has a~well known simple factorization\footnote{The fact that the survival probability has the multiplicative structure~\eqr{Wfac} is equivalent to the absence of positive jumps, by~the strong Markov property; this is the famous ``gambler's winning'' formula \cite{Kyp}.}:
\beq \la{Wfac} \sRui_{\q}^{b}(x,a)  &  =\fr{W_{\q}(x-a)}{W_{\q}(b-a)}.
\eeq

For a~second example, the~\cite{deF} discounted dividends fixed barrier objective for {spectrally negative L\'evy processes}  has a~simple expression in terms of either the $W_\q $ scale function or of its logarithmic derivative $\nu_{\q}=\fr{W'_\q}{W_\q}$ \cite{APP}.\footnote{$\nu_\q$ may be  more useful than $W_\q$ in the \sn \Mar framework \cite{AG}.}
   \begin{equation} V^{b}(x)= \bc \frac{W_\q(  x)}{W_{\q}^ {\prime}(b)}=e^{-\int_{x}^{b} \nu_\q(m)dm}\fr 1{\nu_\q(b)}& 0 \leq x \leq b\\V^{b}(x)=x-b+ V^{b}(b) & x>b \ec. \la{divL} 
\end{equation}

Maximizing over the reflecting barrier $b$ is simply achieved by finding the roots of
\beq \la{smf} W_\q'' (b) =0 \Eq \pd {}{b} \Big[\fr 1{\nu_\q(b)} \Big]=\pd {}{b} \Big[ V^{b}(b)\Big]=1.\eeq


 Since~results for \sn \lev \procs \; (like the \deF \prob \; mentioned above)
require often not much more than the \str, it is natural to attempt to extend them
to the \sn strong Markov case.
As~expected, everything worked out almost smoothly for ``{\lev}-type cases'' like random walks \cite{AV}, Markov additive processes \cite{IP}, etc. Recently, it~was discovered that $W,Z$ formulas continue to hold a~priori for spectrally negative \Mar processes \cite{LLZ17b}, \cite{ALL}. The main difference is that in equations like \eqr{Wfac}, $W_{\q}(x-a)$  must be replaced by a two-variable \fun \ $W_{\q}(x,a)$  (which reduces in the \lev case to $W_{\q} (x,y)=\T W_{\q} (x-y)$, with $\T W_q$ being the scale function of the \lev process). The same holds of course for the second scale function $Z_{\q,\th}(x-a)$ \cite{APP15,IP}. This unifying structure has lead to recent progress for the optimal \divs \ problem for \sn \Mar \procs \ -- see \cite{AG};
however,  as of today, we~are not aware of other  results on the control of the process \eqr{jd} which have succeeded to exploit the $W,Z$ formalism.

 Several approaches may allow handling particular cases of \sn \Mar \procs:

\begin{enumerate}
\item {for the \Seg\ ~ \proc, the~direct IDE solving approach is successful
for computing  the \rp  ---see (\cite{Pau}}) and Theorem 1, Section \ref{s:dir}.

\item{for \lev driven Langevin-type processes, renewal \eqs \ have been provided in \cite{CPRY}} ---see Section \ref{ex:fb}.
\item for processes with  operators having affine drift and volatility, an~explicit integrating factor for the \LT \ may be found in \cite{AU}---see Section \ref{s:AU}.
\item with phase-type jumps, there is Asmussen's embedding into a~regime switching diffusion {{\cite{asmussen1995stationary}}---see Section \ref{s:CLSa}, and~the complex integral representations of {\cite{JJ}}.}\\
\end{enumerate}

We will review and complete here the first approach, and also review and discuss the second  and third approaches.
Asmussen's approach is also recalled, because we believe it has considerable potential.

We end this introduction with  an~example of {a} still open problem we would like to solve in the future:
 \beQ Find the  optimal dividend policy for the \Seg\ process in the presence of capital injections and bankruptcy (in particular, investigate   the extensions of Equations \eqr{divL} and \eqr{smf}).
 \eeQ

{\bf {Contents}}. Section \ref{s:ALL} explains the simplicity of \sn \Mar \procs \ with negative exponential and phase-type jumps, already sketched in \cite{ALL}.

{Section \ref{s:dir} reviews  the direct classic Kolmogorov approach for solving \fp problems. The~discounted \rp ($q>0$)   \eqr{Pau} for the \Seg~ process is obtained, following \cite{Pau}, by transforming the \ren \eqr{reneq} into the ODE \eqr{Paueq}, which~is hypergeometric of order $2$.\footnote{This result due to Paulsen has stopped short further research
for more general
mixed exponential jumps, since it seems to require a~separate ``look-up'' of hypergeometric solutions for each particular
problem.} We also  complete the study of this process by providing its $W$ scale function, using the results in section \ref{s:ALL}.}

 {Section \ref{ex:fb} reviews the recent approach based on renewal
equations due to \cite{CPRY} (which needs still be justified for increasing premiums satisfying \eqr{intc}). An important renewal (Equation \eqr{reneqC}) for the ``scale derivative'' $\w$ is recalled here, and~a~new result relating the scale derivative to the { integrating factor} defined in  \eqr{defI0} is offered---see Theorem 4.}

 {Section \ref{s:AU} reviews older computations of \cite{AU} for more general \procs \ with affine operator, and~provides explicit formulas for the \LTs \ of the \sur and \rp \eqr{seg0q}, in~terms of the same{ integrating factor} \eqr{defI0} and its antiderivative.}

 {Section \ref{s:Segr} checks that our integrating factor approach recovers various results for Segerdahl's process, when $\q=0$ or $x=0$.}

 {Section \ref{s:CLSa} reviews Asmussen's approach for solving \fp problems with
phase-type jumps, and~illustrates the simple structure of the \sur and \rp of the Segerdahl-Tichy process, in~terms of the scale derivative $\w$. This approach yields quasi-explicit results when $q=0$.}

 {Section \ref{s:a} reviews necessary hypergeometric identities (used in section \ref{s:dir}).}

 {Finally, Section \ref{s:con} outlines further promising directions of research.}

\sec{Two-sided \fp \pros \ for \sn \Mar \procs \ with negative exponential and phase-type jumps reduce to computing two one sided  \fp \funs~$H_\q,\Rui_\q$  \la{s:ALL}}
\beR The results in this section apply a priori to a  large class of  \sn \Mar \procs \ \eqr{jd} with negative exponential  jumps, and are formulated as such.  In particular, \BM\ may be present, and hence creeping downwards is possible. \eeR
Our study of  \sn \Mar \procs \ with negative exponential  jumps is based on an increasing
 $q$- harmonic function of our process $H_\q(x), l \leq x $ \satg\ \eqr{H}, and a  decreasing one $ \Rui_{\q}(x,a), x \geq a$,  defined in  \eqr{Rui}.
For the \Seg\ \proc,
  these  functions, to be denoted by $K_1(x), K_2(x), x \geq  0$,  turn out  to be related to the   increasing and decreasing Kummer hypergeometric functions    $M$ and $U$, \resp. Note that $K_1(0)=0$, which renders \eqr{H} immediate. Subsequently, other
 useful functions like $W_q(x,a),Z_q(x,a)$ will be  identified by simple \proba \ arguments which apply a priori to all \procs\ with \expoj, and may be extended to phase-type jumps -- see Lammas 1,2. With $W$ and $Z$ computed, one may hope to solve    complicated control \probs\ involving   dividends and capital injections, by  applying similar arguments as in the \lev case.

 Our first step  is to investigate the existence of a \fac\ formula for two-sided  \fp \pros \ {\bf upwards}, with lower limit at the boundary domain $l$:
\be \sRui_{\q}^b(x,l)=\mathbb{E}_x\left[ e^{-q\tb}; \tb \leq\tl\right]:=
\frac{W_{\q}(x,l_-)}{W_{\q}(b,l_-)}:=\frac{H_{\q}(x)}{H_{\q}(b)}, l \leq  x \leq b,\label{H} \ee
which defines the \fun\ $H_q$, \upm, and up to the existence of the limits when $x\to l_-$ (the latter is a delicate point which must be resolved separately in each particular case).
 The existence of such a function $H$ is suggested by spectral negativity and  the \str. To emphasize that this property holds a priori outside the \lev framework, we provide a justification, based on the trick of adding a point $c >b$, and starting with
\bea \sRui_{\q}^{c}(x,l)=\sRui_{\q}^{b}(x,l) \sRui_{\q}^{c}(b,l),\eea
where only the absence of positive jumps and
the strong Markov property  were used. Therefore,
\bea \sRui_{\q}^{b}(x,l)=\fr{\sRui_{\q}^{c}(x,l)}{ \sRui_{\q}^{c}(b,l)};\eea
So, the quotient decomposition is trivial as long as we stay  on  a fixed interval $[0,c]$, with $c$ arbitrarily big.

 In the \lev case, the dependence on $l$ cancels, and $H_\q(x)=e^{\Fq x}$, where $\Fq$ is the unique \nne\ root of the \CL equation. For diffusions, $H_{\q}(x)$ is the increasing solution of the \SL equation $(\mG-\q) f(x)=0$ (see \fe \cite{BS}). In the general \std\ case, to provide a \fac\  independent of $c$, it suffices  to obtain an increasing
 $q$- harmonic function of our process $H_\q(x),  l \leq x $ \satg\ $H_q(0)=0$, as is the case with the \Seg\ \proc; Doob's optional stopping theorem yields then the \fac.

 When the claims are exponential, computing  \ts \fp \pros\ on $[a,b], a \geq l$, may be reduced to computing $H_\q(x), x \geq l,$ and the \fp ruin \pros\ $\Rui(x,a)$,  cf. \cite{ALL}.
 \beL Let $\Rui_{\q,J}(x,a), \Rui_{\q,C}(x,a)$  denote the killed \rps\ by jump and by creeping, \resp\ (more precisely \LTs, but \LTs are just \rps\ \wr a process where the inter-arrivals are killed-- see \cite{Ivakil} for generalizations and applications to risk theory). For the process \eqr{jd} with negative \expoj, no upward jumps and $a>l$, \ith
 \be\sRui_{\q}^b(x,a)  =\frac{H_\q(x)-\Rui_{\q,C}(x,a)H_\q(a)-\Rui_{\q,J}(x,a)\int_0^{ {a- l}} \mu  e^{-\mu  y}H_\q(a-y)dy}{H_\q(b)-\Rui_{\q,C}(b,a)H_\q(a)-\Rui_{\q,J}(b,a)\int_0^{ {a- l}} \mu  e^{-\mu  y}H_\q(a-y)dy}:=\frac{W_\q(x,a)} {W_\q(b,a)}. \la{We}
 \ee

 Also
\begin{align}
\Rui_{\q}^b(x,a,dy) &=\mathbb{E}_{x}\left[  e^{-q\ta} 1_{\left\{  \ta<\tb , a -X_{\ta} \in dy \right\}  }\right] \no \\&=\Rui_{\q}(x,a,dy) -\sRui_{\q}^b(x,a) \Rui_{\q}(b,a,dy) =\pp{\Rui_{\q}(x,a) -\frac{ W_q(x,a)}{ W_q(b,a)} \Rui_{\q}(b,a)}\; {\mu} e^{-\mu y} dy , \la{Ze}
\end{align}
and \begin{align}
\Rui_{\q}^b(x,a,\th) &=\pp{\Rui_{\q}(x,a) -\frac{ W_q(x,a)}{ W_q(b,a)} \Rui_{\q}(b,a)}\; \fr{\mu} {\mu +\th} . \la{Zet}
\end{align}

\eeL
\beR When $a=l$ this result holds as well, provided that $H_{\q}(l)=0$. \eeR
{\prf}
  \beq \la{12e}
&&\sRui_{\q}^b(x,a)  =\mathbb{E}_{x}\left[  e^{-q\tb}1_{\left\{\tb<\ta\right\}  }\right] \no\\
&&= \mathbb{E}_{x} \left[ e^{-q \tb} \right] -
\mathbb{E}_{x}\left[e^ {-q\ta }1_{\{X_{\ta}=a, \ta<\tb\} }\right]
\mathbb{E}_{a}\left[ e^{-q \tb}\right]\no \\&&
-\int_0^{ {a- l}}\mathbb{E}_{x}\left[e^ {-q\ta }1_{\{a-X_{\ta}\in dy, \ta<\tb\} }\right] \mathbb{E}_{a-y}\left[ e^{-q \tb}\right]\no \\
&&= \sRui_\q^b(x)  -
\Rui_{\q,C}^b(x,a)
\sRui_\q^b(a)
-\Rui_{\q,J}^b(x,a)\int_0^{ {a- l}} \sRui_\q^b(a-y) \mu  e^{-\mu  y} \md y \no\\
&& =\frac 1 {H_\q(b)}
\Big[H_\q(x)  -\Rui_{\q,C}^b(x,a) H_\q(a)
-\Rui_{\q,J}^b(x,a)\int_0^{ {a- l}} H_\q(a-y) \mu  e^{-\mu  y} \md y\Big] \la{int}
\eeq
where $\Rui_{\q,C}, \Rui_{\q,J}$ denote respectively ruin by creeping and by jumps.

Similarly,
\bea  \Rui_{\q,C}^b(x,a)=\Rui_{\q,C}(x,a)-\sRui_{\q,J}^b(x,a)
\Rui_{\q,C}(b,a), \Rui_{\q,J}^b(x,a)=\Rui_{\q,J}(x,a)-\sRui_{\q,J}^b(x,a)
\Rui_{\q,J}(b,a).\eea

  Plugging the last equality into \eqr{int} and putting
  $$W_q(x,a)=H_\q(x)-\Rui_{\q,C}(x,a)H_\q(a)-\Rui_{\q,J}(x,a)\int_0^{ {a- l}} \mu  e^{-\mu  y}H_\q(a-y)dy$$
  yields
  \bea
&&{H_\q(b)} \sRui_{\q}^b(x,a)   =
W_\q(x) +\sRui_{\q}^b(x,a) \pp{H_\q(a) \Rui_{\q,C}(b,a)
+ \Rui_{\q,J}(b,a) \int_0^{ {a- l}} H_\q(a-y) \mu  e^{-\mu  y} \md y}
\eea
and solving for $\sRui_{\q}^b(x,a)$ yields
\beq&&\sRui_{\q}^b(x,a)  =\frac{H_\q(x)-\Rui_{\q,C}(x,a)H_\q(a)-\Rui_{\q,J}(x,a)\int_0^{ {a- l}} \mu  e^{-\mu  y}H_\q(a-y)dy}{H_\q(b)-\Rui_{\q,C}(b,a)H_\q(a)-\Rui_{\q,J}(b,a)\int_0^{ {a- l}} \mu  e^{-\mu  y}H_\q(a-y)dy}. \la{We}
 \eeq

 \beR
\eqr{12e}, \eqr{Ze} show that,  with  negative \expo \ jumps, both \ts \fp \pros \ may be constructed using three functions  $H_\q, \Rui_{\q,C}, \Rui_{\q,J}$ from the one-sided theory. If down-crossing continuously is impossible, only two functions  $H_\q,  \Rui_{\q,J}$  are necessary.

The extension to downwards jumps of phase-type $(\vec \b, B)$ (a dense family) is immediate. \eeR
\beL  For \procs \ with downwards jumps of phase-type $(\vec \b, B)$ \eqr{12e} becomes:
\beq
&&\sRui_{\q}^b(x,a)  =\mathbb{E}_{x}\left[  e^{-q\tb}1_{\left\{\tb<\ta\right\}  }\right] \no \\&&=\frac{H_\q(x)-\Rui_{\q,C}(x,a)H_\q(a)-\vec \Rui_{\q,J}(x,a)\int_0^{ {a- l}}   e^{B  y} \bff b H_\q(a-y)dy}
{H_\q(b)-\Rui_{\q,C}(b,a)H_\q(a)-\vec \Rui_{\q,J}(b,a)\int_0^{ {a- l}}   e^{B  y} \bff b H_\q(a-y)dy}:=\frac{ W_q(x,a)}{ W_q(b,a)}, \la{WeP}
 \eeq 
 where $\vec \Rui_{\q,J}(x,a)$ is the lign vector of \rps \ whose $k$-th component is the \rp \ when crossing of $x$ axis occurs during phase $k$, and $\bff b=(-B) \bff 1$ .

 Similarly, \eqr{Ze} becomes
 \be \Rui_{\q}^b(x,a,dy) = \pp{\vec \Rui_{\q}(x,a) -\sRui_{\q}^b(x,a) \vec \Rui_{\q}(b,a)} e^{ B y} \bff b dy . \la{ZeP} \ee
 \eeL
 \prf The same ideas as in the \expo case apply, except that
 now we must take into account the ``conditional memory-less property  of phase-type variables":
 \beq \la{12eP}
&&\sRui_{\q}^b(x,a)  = \sRui_\q^b(x)  -
\Rui_{\q,C}^b(x,a)
\sRui_\q^b(a)
-\int_0^{ {a- l}}\sum_{k=1}^K \mathbb{E}_{x}\left[e^ {-q\tb }1_{\{\ta<\tb, J_c=k,X_{\ta}\in a-dy, \tb <\I\} }\right],\no \eeq
where $J_c$ is the phase when down-crossing $a$.  Now the last term may be written as
\bea&&\sum_{k=1}^K \int_0^{ {a- l}} \mathbb{E}_{x}\left[e^ {-q\tb }1_{\{\ta<\tb, J_c=k,X_{\ta}\in a-dy, \tb <\I\} }\right]=\\&&\sum_{k=1}^K \int_0^{ {a- l}} \mathbb{E}_{x}\left[e^ {-q\ta }1_{\{\ta<\tb, J_c=k,X_{\ta}\in a-dy\}} \right] \mathbb{E}_{a-y}\left[ e^{-q \tb}; \tb <\I \right]   \\
&&\sum_{k=1}^K \Rui_{\q,J}(x,a,k) \int_0^{ {a- l}} \pp{e^{B y} \bff b}_k    \sRui_\q^b(a-y)  \md y=
\vec \Rui_{\q,J}(x,a)\int_0^{ {a- l}}   e^{B  y} \bff b \fr{H_\q(a-y)dy}{H_\q(b)},\eea
where $\Rui_{\q,J}(x,a,k)$ denotes the \rp with crossing in phase $k$,
and where the conditional memory-less property was applied.

The rest of the proof must be modified similarly. \qed
 \beR

 Note that the  formula
$$H_\q(x)=  W_q(x,a)+\Rui_{\q,C}(x,a)H_\q(a)+\vec \Rui_{\q,J}(x,a)  \int_{0 }^{ {a- l}}     e^{B  z} \bff b H_\q(a-z)dz$$ has a clear heuristic probabilistic interpretation:  {the ``total weight", starting from $x$, of all paths converging to $\I$  equals the ``total weight" of all paths not reaching $a$ + the ``total weight" of all paths  dropping to some $a-z >l, 0 \leq z$,  and converging to $\I$ afterwards}. Note  that in the presence of a  lower limit $l$,  converging to $\I$ in the heuristic may be replaced by never reaching $l$.

\eeR

\section{Direct Conversion to an~ODE of Kolmogorov's Integro-Differential Equation for the Discounted Ruin Probability} \la{s:dir}

One may associate to the process \eqr{jd}
 a~Markovian semi-group with generator
\begin{equation*} \label{genrisk}
\mG h(x) = v(x) h''(x) + c(x) h'(x) + \int_{(0,\I)}
[h(x-y) - h(x) ] \lm(d y)
\end{equation*}
acting
on $h\in C^2_{(0,\I)}$, up to the minimum between its explosion and exit time {$T_{0,-}$}. 

The classic approach for computing the ruin, survival, optimal dividends, and other similar functions
 starts with the \wk \ Kolmogorov integro-differential equations associated to this operator. With  jumps having a rational \LT, one may remove the integral term in \Kol's \equ above by applying to it the differential operator $n(D)$ given by the denominator of the \LE $\k(D)$. For example, with \expo claims, we would apply the operator
 $\mu +D$.

 \ssec{Ruin probabilities for Segerdahl's Process with Exponential Jumps {Paulsen and
  Gjessing} (1997), ex. 2.1 \la{s:PG1}}

 When $v(x)=0$ and $C_k$ in \eqr{jd} are exponential i.i.d random variables with density $f(x)=\mu e^{-\mu x}$, the Kolmogorov integro-differential \eq \ for the \rp \ \ is:
 {\beq && \la{reneq} c(x) \Rui_\q(x,a)' + \lambda \mu \int_a^x e^{-\mu(x-z)} \Rui_\q(z,a) dz  -(\lambda +q) \Rui_\q(x,a)+ \l e^{-\mu x}=0, \no \\&& \Rui_\q(b,a)=1, \Rui_\q(x,a)=0, x<a.\eeq}

To remove the convolution term $\Rui_\q* f_C$, apply the operator
 $\mu +D$, which~replaces the convolution term by $\l \mu \Rui_\q(x)$\footnote{More generally, for any \PH jumps $C_i$ with \LT $\H f_C(s)=\fr{a(s)}{b(s)}$, it may be checked that $\Rui_\q* f_C=\H f_C(D) \Rui_\q$ in the sense that $b(D) \Rui_\q* f_C=a(D) \Rui_\q$,~thus removing the convolution by applying the denominator
 $b(D)$.} yielding finally
 \bea \le(c(x) D^2 + (c'(x) + \mu c(x)-(\lambda + \q)) D - \q \mu)\ri) \Rui_\q(x) =0\eea

When $c(x)=c+ r x, a=0, b=\I$, the \rp \sats:
\begin{eqnarray} \la{Paueq}
&&\le[(\Tc + x) \,D^2 + (1+ \mu (\Tc +x)  -\Tq - \Tl) D - \mu
\Tq\ri]
 \Rui_{q}(x)  =0, \no \\&&(-c  \,D +   \l +\q)
 \Rui_{q}(0)  =\l\footnote{this is implied by the Kolmogorov integro-differential equation $(\mG -\l - q) \Rui_\q(x)+ \lambda \ovl F(x)=0, x \geq 0$}, \quad \Rui_{q}(\I)=0 \la{bc} \eeq
 see \cite[{(2.14)},{(2.15)}]{Pau}, where $ \Tl =\fr{\lambda}r,
\Tq =\fr{\q}r$, and~$-\Tc:=-\fr{c}r$ is the \abs ruin~level.

Changing the origin to $-\Tc$ by $z=\mu(x+ \Tc),  \Rui_{q}(x)=y(z)$ brings this to the
 form
 {  \begin{equation} \la{conhypeq}
 z y''(z) + (z+1-n) y'(z) - \Tq y(z)=0, \; n= \Tl + \Tq,   \ee}

\noindent(we corrected here two wrong minuses in \cite{Pau}),
 which corresponds to the process killed at the absolute ruin, with claims rate $\mu=1$.
 Note that the (Sturm-Liouville) Equation~\eqr{conhypeq} intervenes also in the study of the squared radial \OU diffusion (also called \CIR process) \cite[p. 140, Chapter II.8]{BS}.

 { Let $K_i(z)=K_i(\Tq,n,z), i=1,2, n=\Tq + \Tl$ denote the
 (unique up to a~constant) increasing/decreasing solutions for $z\in (0,\I)$ of the confluent hypergeometric Equation \eqr{conhypeq}. The~solution of \eqr{conhypeq} is thus}
 {  \beq \la{Ki}  c_1 K_1(\Tq,n, z) + c_2 K_2(\Tq,n, z)=c_1 z^n e^{-z}M(\Tq+1,n+1,z)  +c_2 z^n e^{-z}
   U(\Tq+1,n+1,z),\eeq}

\noindent where \cite[{13.2.5}]{AS}
$U[a, a+c, z] = \fr 1{\Gamma[a]} \int_0^\I e^{-z t} t^{a - 1} (t + 1)^{c - 1} dt,  Re[z] > 0, Re[a] > 0$ is Tricomi's decreasing hypergeometric U function
and
   $M\le(a,a+c,z\ri)= \, _1 {F}_1 (a,a+c;z)
 $ is Kummer's increasing \nne \ confluent hypergeometric function of the
first kind.\footnote{$M(a,b,z)$ and $U(a,b,z)$ are the increasing/decreasing solutions of the 
to
 Weiler's canonical form of Kummer equation $z f''(z) + (b-z) f'(z) - a~f(z)=0$, which~is obtained via the substitution $y(z) =e^{-z} z^{n} f(z),$ with $a=\Tq +1, b=n +1$. Some computer systems use instead of $M$ the Laguerre function defined by $M(a,b,z)=L_{-a}^{b-1}(z)\fr{\Gamma(1-a) \Gamma(b)}{\G(b-a)}$, which~yields for natural $-a$ the Laguerre polynomial of degree $-a$.}

  The killed ruin probability must be  combination of $U$ and $M$, but the fact that it decreases to $0$
suggests the absence of the function $K_1$. The~next result shows that this is indeed  the case: the \rp is proportional to $K_2(\mu(x+ c(a)/r))$   on an~arbitrary interval $[a,\I), a>-\Tc$.
$K_1$ yields the absolute \sur \pro \ (and scale function) on $[-\Tc,\I)$, but over an~arbitrary interval we must use a~combination of $K_1$ and $K_2$.

\begin{Thm} \la{t:W}

  Put $z(x)=\mu (\Tc +x), \Tc=c(a)/r$. The~\rp on $[a,\I)$ is
{\begin{equation}\label{Pau}
\Rui_{q}(x,a) =E_x [e^{-q \ta}]=\fr{\Tl}{\Tc \mu}  \frac{e^{-\mu
  x }(1+ x/\Tc)^{(\Tq+\Tl) }\; U\big(1+ \Tq, 1+ \Tq+\Tl,\mu (\Tc +x)\big)}
{U\big(1+ \Tq, 2+ \Tq+\Tl,\mu \Tc \big)}
\end{equation}}

\noindent (when $q=0,$ $K_2(0,n,z)=\G(\Tl,z)$ and \eqr{Pau} reduces to $\fr{\G(\Tl,\mu (\Tc +x))}{\G(\Tl+1,\mu \Tc ))})$.\footnote{Note that we have corrected Paulsen's original denominator by using the identity \cite[{13.4.18}]{AS}
 $ U[a-1, b, z]+(b-a) U[a, b, z]=z U[a, b+1,z], a>1.$}
\end{Thm}

\begin{proof}
  Following \cite[ex. 2.1]{Pau}, note that the limits $\lim_{z\to \I} U(z)=0, \lim_{z\to \I} M(z)=\I$ imply
  \bea \Rui_{q}(x)= k \; K_2(z)=k {e^{-
z} z^{\T \q + \Tl} U\left( \T \q+1%
,\T \q +\Tl+1,z\right) },  \qu z=\mu(x + \Tc).\eea

The proportionality constant $k$ is obtained from the boundary condition \eqr{bc}.
Putting $G_b [h](x):= [c(x) (h)'(x)- (\lambda + q) h(x)]_{x=0},$
 \bea  G_b [ \Rui_q](x)+\lambda=0 \Lra k=\fr{\l}{-G_b [K_2](z(x))},  \eea
\bea  =-e^{-z} z^{\Tq + \Tl+1}U\left(
\Tq+1,\Tq +\Tl+2,z \right)\eea

Putting $z_0= \mu c$, we find
\bea &&-G_b [K_2](z(x))={z_0 e^{-z_0}
z_0^{\T \q + \Tl-1} U \le(\T \q, \T \q + \Tl, z_0\ri)  +(\T \q + \Tl)e^{-z_0}
z_0^{\T \q + \Tl}U \le(\T \q+1, \T \q + \Tl+1, z_0\ri)}\\&&={ e^{-z_0}
z_0^{\T \q + \Tl} (U \le(\T \q, \T \q + \Tl, z_0\ri)  +(\T \q + \Tl)U \le(\T \q+1, \T \q + \Tl+1, z_0\ri))}, \eea
where we have used the
identity \cite[p. 640]{BS}
  \beq
 K_2'(z)&=&- e^{-z}
z^{\T \q + \Tl-1} \
 U \le(\T \q, \T \q + \Tl, z\ri). \la{Kids}\eeq

 This may be further simplified to
 \bea &&-G_b [K_2](z(x))={ e^{-z_0}
z_0^{\T \q + \Tl+1}  U \le(\T \q+1, \T \q + \Tl+2, z_0\ri))},\eea
by using the identity
 \beq \la{sPau} U[a, b, z]+b U[a+1, b+1, z]=z U[a+1, b+2,z], a>1,\eeq
 which is itself a~consequence of the identities \cite[{13.4.16,13.4.18}]{AS}
 \beq \la{abr} &&(b - a) U[a, b, z] +
 z  U[a, 2 + b, z] = (z + b) U[a, 1 + b,
   z]\\&&U[a, b, z]+(b-a-1) U[a+1, b, z]=z U[a+1, b+1,z]\eeq
 (replace $a$ by $a+1$ in the first identity, and~subtract the second).

 Finally, we
obtain:
\begin{eqnarray*} \label{Pauruin}
&&\Rui_{q}(x)= \left( \frac{\Tl}{\Tc \mu }\right) \frac{e^{-
\mu  x} (1+ \fr{x}{\Tc} )^{\T \q + \Tl} U\left( \T \q+1%
,\T \q + 1+\Tl,\mu  (x + \Tc )\right) }{U\left( \T \q+1,\T \q +1+
\Tl+1,\mu \Tc  \right) }\\&&=\left( \frac{\Tl}{ \mu }\right)
\frac{\int_x^{\infty} ({s-x})^{\T \q} \, ({s}+{\Tc} )^{ \Tl-1} \,
 e^{- \mu  s}
 {d s} }{\int_{0}^{\infty}~s^{\T \q}~(s+ \Tc)^{ \Tl}~e^{ - \mu  s } ds  }
,\no\end{eqnarray*}
and
\begin{align*}
&&\Rui_{q}(0)  =\left( \frac{\lambda}{c\mu}\right)  \frac{U\left(
\T \q+1 ,\Tl + \T \q+1,\Tc \mu\right) }{U\left(\T \q+1,\Tl + \T \q + 2,\Tc
\mu\right)  }=\left( \frac{\lambda}{c\mu}\right)
\frac{\int_{0}^{\infty}%
\,t^{\T \q}~(1+t)^{
\Tl -1}\,e^{-\Tc \mu t }~dt}{\int_{0}^{\infty}\,\,\,t^{\T \q}%
~(1+t)^{\Tl} \,e^{-\Tc \mu t } ~dt}\\&&=\left( \frac{\Tl}{ \mu
}\right) \frac{\int_0^{\infty} {s}^{\T \q} \, ({s}+{\Tc} )^{ \Tl-1} \,
 e^{- \mu  s}
 {d s} }{\int_{0}^{\infty}~s^{\T \q}~(s+ \Tc)^{ \Tl}~
 e^{ - \mu  s } ds  } =\left( \frac{\Tl}{\Tc}\right)
\frac{\int
_{0}^{\infty}\,t^{\T \q}%
~(t+\mu)^{\Tl-1}\,e^{-\Tc t}~dt}{\int_{0}^{\infty}\,\,\,t^
{\T \q}~(t+\mu)^{\Tl} \,e^{-\Tc t}%
dt}.
\end{align*}

For $(a,\I), a>-\Tc$, the~same proof works after replacing $c,z(0)$ by $c(a),z(a)$.

\end{proof}

\ssec{Essentials of \fp theory for the \Seg\ \ \proc \la{s:fpS}}
We  gather  now together the most basic \fp results for  the  \Seg\ \ process with $\r=1$ and  $\c=0$ (so that the absolute ruin point is $-\Tc=0$), and $ \mu=1.$
 The general case with $\c \neq 0, \r \neq 0,  \mu \neq 1,  $ can be obtained by replacing $z,a$ in the theorem below by $z(x):=\mu (x+ \fr{\c}{c_1})$ and  $z(a)$ -- see Section \ref{s:PG1}.
\beT \la{tK}When $\r =1, \c = 0,$ the following formulas hold, with $n=\q+\l$:
 \BEN \im
  The function $ H_\q$   is up to a \prc
 \be
 H_\q(z)=K_1(z)\sim z^n e^{-z}  M(\q+1,n+1,z)=z^n M(\l ,n+1,-z), \ee
   where $K_1(z)$ is the  unique
$q$-harmonic function which increases on $(0,\I)$. The last expression, obtained via a Kummer transformation, is sometimes more stable numerically.

 \im For $a  >0$, the ruin function is
 \be \Rui_{\q}(z,a )=E_z [e^{-q T_{a ,-}}]=\l  \fr{K_2(\q,n,z)}{K_2(\q,n+1,a  )} = \bc \fr{\l }{\mu  } \fr{e^{-z} z^n U(\q+1,n+1,z)}{e^{-a  } a ^{n+1} U(\q+1,n+2,a  )},  & z \geq a  >0\\\frac{\l  \Gamma \left(1+\q\right) }{ \Gamma \left(1+\q +\l\right)}e^{-z} z^n U(\q+1,n+1,z),&a=0\ec
 \la{Pau}\ee
 ({where we used $\lim_{z \to 0} z^{\q + 1 + \l}  U[1 + \q, \q + 2 + \l,z] =\frac{\Gamma \left(1+\q +\l\right) }{ \Gamma \left(1+\q\right)}$} for the case $a=0$).

\im   $$\Rui_{\q}(z,a ,\theta)=E_z [e^{-q T_{a ,-}+\th X_{\ta}}]=\Rui_{\q}(z,a )\fr{\mu}{\mu + \th}$$
 (by the memoryless property  of \expo \ claims).

\im For $z \geq  a  > 0$, cf.   \eqr{We}, the~scale function  is  given by
 \beq &&  \la{WSeg} W_q(z,a )=H_\q(z)-\Rui_{\q}(z,a )\int_0^{a } \mu e^{-\mu y}H_\q(a -y)dy=K_1(q,n,z)-\Rui_{\q}(z,a ) \fr{K_1(\q,n+1,a )}{n+1}\no \\&&=K_1(q,n,z)-\fr{\l}{n+1} \fr{K_1(\q,n+1,a )}
 {K_2(\q,n+1,a )} K_2(q,n,z)\\\no
 &&=z^n e^{-z} \pp{ M(\q+1,n+1,z)-\fr{\l }{n + 1} U(\q+1,n+1,z) \, \fr{M(q+1,n+2,a  ) }{  U(\q+1,n+2,a  )}  }.
\eeq

Since this is only determined \upc, we may and will usually take \fno
\beq &&W_\q(z,a )= z^n e^{-z}\begin{vmatrix} M(1+\q,\l+1+\q,z) & M(1+\q,\l+2+\q,a  ) \\ \l  U(1+\q,\l+1+\q,z) & (\l  + 1+\q)  U(1+\q,\l+2+\q,a  ) \end{vmatrix}=\\&& \no  z^n e^{-z}\pp{
{(\l +\q + 1) M(\q+1,n+1,z) U(\q+1,n+2,a  )-\l  U(\q+1,n+1,z) \, M(q+1,n+2,a  ) }}. \la{Wee}\eeq

 The second derivative is
\bea &&{W_\q''(z,a)}=\mu^2 \pr{K_1''(q,n,z)-\fr{\l}{n+1} \fr{K_1(\q,n+1,a )}
 {K_2(\q,n+1,a )} K_2''(q,n,z)} \Lra {W_\q''(z,a)} e^{a} a^{-n+2} / \Gamma (n) \\&&=\frac{\,
   _1\tilde{F}_1(q+1;n+2;a) \left(\frac{\lambda  n \left(a
   n-a q-n^2+n\right) U(q+1,n+1,a)}{U(q+1,n+2,a)}+a (a-n+1)
   ((n+1) (n-q)-\lambda  n)\right)}{n+1}+\\&&\left(-(a+1) n+a
   q+n^2\right) \, _1\tilde{F}_1(q+1;n+1;a)\eea
 and \sats $$W_\q''(0,0) =\frac{(n-1) n \left(1-\frac{\lambda  n U(q+1,n+1,0)}{(n+1)^2
   U(q+1,n+2,0)}\right)}{\Gamma (n+1)}<0$$
   iff
   \be \q + \l <1.\ee

\im  The two-sided ruin function \eqr{Rui} with stopping at an upper bound $b$ \sats
 \bea &&\Rui_{\q}^b(z,a )=\Rui_{\q}(z,a )  -\fr{W_\q(z,a )}{W_\q(b,a )}  \Rui_{\q}(b,a )=
 \fr{\l }{a } \fr{e^{-z} z^n U(\q+1,n+1,z)}{e^{-a} a ^n U(\q+1,n+2,a )} -\fr{\l }{a } \fr{e^{-b} b^n U(\q+1,n+1,b)}{e^{-a} a ^n U(\q+1,n+2,a )}  \times \\
 &&
 (\frac z b)^n e^{b-z}\fr{(\l +\q + 1) M(\q+1,n+1,z) U(\q+1,n+2,a )-\l  U(\q+1,n+1,z) \, M(q+1,n+2,a ) }
 {(\l +\q + 1) M(\q+1,n+1,b) U(\q+1,n+2,a )-\l  U(\q+1,n+1,b) \, M(q+1,n+2,a ) }. \la{Zee}
\eea

  \EEN
\eeT

 \begin{proof} 1. holds since $K_1(z)$ is the unique
solution which increases on $(0,\I)$.

2. is a particular case of Theorem 1.

 3. This follows from the memoryless property of the \expoj.

 4. Apply   \eqr{We}. \Mp
\bea &&   W_q(z,a )=H_\q(z)-\Rui_{\q}(x,a )\int_0^{a }  e^{y-  a }H_\q(y)dy=K_1(z,n)-\Rui_{\q}(x,a ) \fr{K_1(a,n+1)}{n+1}\\&&=z^n e^{-z}  M(\q+1,n+1,z)-{\l } \fr{e^{-z} z^n U(\q+1,n+1,z)}{ a ^{n+1} U(\q+1,n+2,a )} \int_0^{a }  y^n M (
\q+ 1,n + 1,y )dy \no
\\&&=z^n e^{-z}  M(\q+1,n+1,z)-{\l } \fr{e^{-z} z^n U(\q+1,n+1,z)}{ a ^{n+1} U(\q+1,n+2,a )} \fr{a ^{n+1}}
{n+1} \, M(q+1,n+2,a) \no
\\&&=z^n e^{-z} \pp{ M(\q+1,n+1,z)-\fr{\l }{n + 1} U(\q+1,n+1,z) \, \fr{M(q+1,n+2,a) }{  U(\q+1,n+2,a )}  },
\eea
where we used $\int_0^{a }  y^n M (
\q+ 1,n + 1,y )dy=\fr{a ^{n+1}}
{n+1} \, M(q+1,n+2,a)$ -- see \cite{AS}.

5. This result is immediate.
 \end{proof}

 \beR   
 The apparent singularity in \eqr{Pau} when $a,z\to 0$ may be removed, since \be \la{ru00} \Rui_{q}(0,0)=\lim_{z->0} \Rui_{q}(z,z) =\lim_{z->0} \left( \frac{\l }{z}\right)  \frac{U\left(
\q +1 ,\l  + \q +1,z\right) }{U\left(\q +1,\l  + \q  + 2,z\right)  }=\frac{\l   \Gamma (\q+\l  )}{\Gamma (\q+\l  +1)}=\frac{\l }{\l  + \q}.\ee

This result has a clear probabilistic interpretation and holds in fact
clearly for any \Lm of finite negative intensity $\l$.

\eeR

  \beR
When $\q=0$, \eqr{Wee} and
\beq \bc M(1,1+\l,z)= \lambda  e^x x^{-\lambda } \g(\l,z)\\ U(1,1+\l,z)= e^x x^{-\l} \G(\l,z)=e^x E_{1-\lambda }(z) \ec \la{MUG}\eeq
where  $\G(\l,y)=\int_y^\I {x^{\l -1}} e^{-x} dx$ is the incomplete gamma function,  $\g(\l,y)=\int_0^y {x^{\l -1}} e^{-x} dx=\G(\l)-\G(\l,y)$  is the lower incomplete gamma function,
    and $E_{\lambda }(z)=\int_1^\I t^{-\l} e^{-x t} dt$ is the ExpIntegral function, yield
$$W(z,a )= z^{\l} e^{-z} \begin{vmatrix} M(1,\l+1,z) & M(1,\l+2,a ) \\ \l  U(1,\l+1,z) & (\l  + 1)  U(1,\l+2,a ) \end{vmatrix}=e^a \left(-a^{-\lambda -1}\right) \Gamma (\lambda +2)\lambda (  \Gamma (\lambda
   ,z)-\Gamma (\lambda,a )).$$

Up to a constant, we have $$W(z,a ) \sim  \Gamma (\lambda,a )-\Gamma (\lambda
   ,z)=\int_a^z t^{\l -1} e^{-t} dt \Lra W'(z,a ) \sim   e^{-z} z^{\l -1},    $$ a particular case of the  formula
   $W'(z,a ) \sim   e^{-z} c(z)^{\l -1},$
   which will be rederived below.
\eeR

\sec{{The Renewal Equation for the Scale Derivative of \lev Driven Langevin Processes \la{ex:fb} {Czarna~et~al.}~(2017)}}

 One tractable extension of the \Seg-Tichy process is provided by {is} the ``Langevin-type'' risk process
defined by
\begin{equation}  \la{Lang} X_t=x+\int_0^t c(X_s)\;ds + Y_t, 
\end{equation}
where $Y_t$ is a~spectrally negative \lev \proc, 
and $ c(u)$  is a~\nne
premium function    {satisfying}
\be \la{intc} u >0 \Lra c(u) >0, \;   \int_{x_0}^\I\frac{1}{c(u)}\,du = \I,  \; \for x_0  >0.\ee

The integrability condition above is necessary to preclude explosions.
 Indeed when $Y_t$ is a compound Poisson process, in~ between jumps (claims) the risk process \eqr{Lang} moves
deterministically along the curves $x_t$ determined by the vector field
\bea \la{vf}\fr {dx}{dt}=c(x) \Eq t=\int_{x_0}^{x} \fr {du}{c(u)}:=C(x;x_0), \for x_0 >0.\eea

From the last equality, it may be noted
that if $C(x;x_0)$ satisfies $\lim_{x \to \I} C(x;x_0)< \I$, then $x_t$ must explode, and~the stochastic process $X_t$ may explode.

The case of Langevin processes 
has been tackled recently in \cite{CPRY}, who provide the construction of the process \eqr{Lang} in the particular case of { non-increasing} functions $c\cd$. This setup can be used to model dividend payments, and~other mathematical finance applications.

\cite{CPRY} showed that the $W,Z$ scale \funs \; which provide a~basis for first passage problems
 of \lev \sp negative processes have two variables extensions $\mW,\mZ$ for the process \eqr{Lang}, which~satisfy integral equations. The~equation for $\mW$, obtained by putting $\phi(x)=c(a) -c(x)$ in  \cite[eqn.~(40)]{CPRY}, is:
 \beq \la{reneqg} &&\mW_{\q}(x,a)  =W_{\q}(x-a)  + \int_a^x (c(a)-c(z)) W_{\q}(x-z) \mW_{\q}'(z;a) dz,
 \eeq
where $W_{\q}$ is the scale function of the \lev process
 obtained by replacing $c(x)$ with $c(a)$.

It follows that the { scale derivative}
 \bea 
 \w_{\q}(x,a)=\pd{}{x}\mW_\q(x,a)\eea of the scale function of the process \eqr{Lang}
 \sats \ a~Volterra renewal equation \cite[eqn.~(41)]{CPRY}:
 \be \la{reneqw} \w_{\q}(x,a)\le({1+ (c(x)-c(a)) W_\q(0)} \ri)  ={w_{\q}(x-a)}
   + \int_a^x (c(a)- c(z)) w_{\q}(x-z) \w_{\q}(z;a) dz,
 \ee
 where $w_{\q}$ is the derivative of the scale function  of the \lev process
 $Y_t=Y_t^{(a)}$ obtained by replacing $c(x)$ with $c(a)$.
 This may further be written as:

  \beq \la{reneqC} && w_{\q}(x-a)  + \int_a^x w_{\q}(x-z) \w_{\q}(z;a) (c(a)-c(z)) dz =\bc \w_{\q}(x,a),& Y_t  \text{ of unbounded variation}\\
 \w_{\q}(x,a)\fr {c(x)}{c(a)},& Y_t \text{ of bounded variation}\ec.
 \eeq

\beQ It is natural to conjecture that the formula \eqr{reneqC} holds for all drifts satisfying \eqr{intc}, but this is an~open problem for now.
\eeQ

\beR\label{rem_lt} Note that renewal equations are a~more appropriate tool than Laplace transforms for the general Langevin problem. Indeed, taking ``shifted \LT" $\mL_a f(x)=\int_a^\I e^{-s (y-a)} f(y) dy$ of \eqr{reneqC}, putting
   \bea \bc \H \w_{\q}(s,a)=\int_a^\I e^{-s { (y-a)}} \w_{\q}(y,a) dy,\\ \H \w_{q,c}(s,a)=\int_a^\I e^{-s{ (y-a)}} \w_{\q}(y,a)  c(y)dy \\ {\H w_{\q}(s)=\int_0^{\infty}e^{-s y} w_{\q}(y)dy}\ec,\eea
 and using $$\mL_a [\int_a^x f(x-y) l(y) dy](s)=\mL_0 f(s) \mL_a l(s)$$ yields equations with two unknowns:
 \beq  \la{LTCPRY} \H w_{\q}(s)(1+c(a)\H \w_{\q}(s,a)-\H \w_{q,c}(s,a))=\bc \H \w_{\q}(s,a)& \text{ unbounded variation case}\\ \fr{\H \w_{\q,c}(s,a)}{c(a)} & \text{ bounded variation case} \ec,\eeq
 whose solution is not obvious.

\eeR
\ssec*{{The Linear Case} $c(x)=rx+c$}
To get explicit \LTs, we will turn next to \OU type processes\footnote{For some background first passage results on these processes, see \fe  \cite{borovkov2008exit,LoefPat}.} $X\cd$, with $c(x)= c(a) + r (x-a)$,
which implies \beq\label{aux_jl_2}
 \H \w_{\q,c}(s,a)=\int_a^\I e^{-s  {(y-a)}} \w_{\q}(y,a) (r {(y-a)}+c(a)) dy=-r\H \w_{\q}'(s,a)+c(a) \H \w_{\q}(s,a).
 \eeq

Equation \eqr{LTCPRY} simplify then to:
 \beq  \la{AUCPRY} \H w_{\q}(s)(1+r\H \w_{\q}'(s,a))=\bc \H \w_{\q}(s,a)
 & \text{ unbdd variation case}\\\H \w_{\q}(s,a)-\fr{r}{c(a) }\H \w_{\q}'(s,a)& \text{ bdd variation case} \ec.\eeq

\beR Note that the only dependence on $a$ in this equation is via $c(a)$, and~via the shifted \LT. Since $a$ is fixed, we may and will from now on  simplify by assuming w.l.o.g. $a=0$, and~write $c=c(a)$. \eeR

Let now
 $$ \k(s)= \a_0 s^2 + c s- s \H {\ovl \lm}(s) -q, \a_0 >0, $$ denote the \LE\ or symbol of the \lev process $Y_t=\sqrt{2\alpha_0}B_t-J_t+ct$,
and recall that
\bea w_{\q}(s)=\bc \fr s{\k(s)}&\text{ unbdd variation case}\\
\fr s{\k(s)}-\fr 1 c&\text{ bdd variation case}\ec \eea
 (where we have used that $W_\q(0)=0(\fr 1 c)$ in the two cases, respectively).

We obtain now from \eqr{AUCPRY}
the following ODE
\beq\label{odeCPRY}
r\H \w_{\q}'(s,a)-\frac{\k(s)}{s}\H \w_{\q}(s,a)=-1+\fr{\k(s)}{ s}W_\q(0)=\bc-1&\text{ unbdd variation case}\\-1+\fr{\k(s)}{c s}:=-\fr{h(s)}{c}&\text{ bdd variation case}\ec,
 \eeq
where $$h(s)=\H {\ovl \lm}(s) +\fr \q{ s}.$$

 \beR \la{r:int}
 The Equation \eqr{odeCPRY} is easily solved multiplying by an~integrating factor
\beq \la{defI0} I_{q}(s,s_0)=e^{-\int_{s_0}^{s}\,\frac{ \k(z)/z}r dz }=e^{-\int_{s_0}^{s}\,\frac{\a_0 z+c -\woF
(z)-\q/z}r dz }, \eeq where $s_0 > 0$ is an~arbitrary integration limit chosen so that
the integral converges (the formula \eqr{defI0} appeared first in \cite{AU}).
To simplify, we may choose $s_0=0$ to integrate
the first part $\a_0 z+c -\woF(z)$, and~ a~different lower bound $s_0=1$ to integrate $q/z$.
Putting $\Tq=\fr q r, \Tc=\fr c r, \T \a_0 =\fr {\a_0} r$, we~get~that
\begin{equation} \label{defI}
I_{\q}(s)=e^{-\int_{\cdot}^{s}\,\frac{ \k(z)/z}r dz }= s^{\Tq }
e^{-\left[  \left(  \frac{\T \a_0}%
{2}\right)  s^{2}+\Tc s\right]  + \frac{1}{r \, } \int_0^s \woF(z) dz} := s^{\Tq } I(s):=e^{-\Tc s} i_{\q}(s),
\end{equation}
where we replaced $s_0$ by $\cdot$ to indicate that two different lower bounds are in fact used, and~we put $I(s)= I_0(s)$ (the subscript $0$ will be omitted when $q=0$).
\eeR

 Solving \eqr{odeCPRY} yields:
\beT \la{cor:wI} Fix $a$ and put $\ovl I_{\q}(s)= \int_s^\I   I_{\q}(y) dy$. Then, the \LT \ \ of the \sd of an~\OU type process \eqr{Lang} satisfies:

 \begin{equation}  \la{wI}  \H \w_{\q}(s,a) = \fr{ \ovl I_{\q}(s) }{r I_{\q}(s)}- W_\q(0)=\bc \fr{ \ovl I_{\q}(s) }{r I_{\q}(s)},&
	\text{ in the unbounded variation case}\\
 \fr{ \ovl I_{\q}(s) }{r I_{\q}(s)} -\fr 1 c,&\text{ in the bounded variation case}\ec.\end{equation}

\eeT

\begin{proof}
{In the unbounded} 
variation case, applying the integrating factor to \eqr{odeCPRY} yields immediately:

\bea \no &&\H \w_{\q}(s,a) I_{\q}(s)= r^{-1} \int_s^\I {I_{\q}(y)} dy=r^{-1} \ovl I_{\q}(s).\eea

In the bounded variation case, we observe that $$i_{\q}'(s)=\fr{h(s)}r i_{\q}(s),$$ where $i_q$ is defined in \eqr{defI}. An integration by parts now yields
\bea \no &&\H \w_{\q}(s,a) I_{\q}(s)= \int_s^\I \fr{h(y)}{c r} {I_{\q}(y)} dy=\int_s^\I \fr{h(y)}{c r} e^{-\Tc y} i_{\q}(y) dy\\&&=c^{-1} \int_s^\I   e^{-\Tc y} i_{\q}'(y) dy=c^{-1}(-I_{\q}(s)+\Tc \int_s^\I   e^{-\Tc y} i_{\q}(y) dy)=r^{-1} \ovl I_{\q}(s) dy-c^{-1} I_{\q}(s).\eea
\end{proof}

\beR The result \eqr{wI} is quite similar to the \LT \ for the \sur and \rp\ (Gerber-Shiu functions) derived in \cite[p. 470]{AU}---see \eqr{AU}, \eqr{seg0q} below; the main difference is that in that case additional effort was needed for finding the values $\sRui(a,a), \Rui(a,a)$. \eeR

\sec{The \LT -Integrating Factor Approach for Jump-Diffusions with Affine Operator 
{Avram and Usabel} (2008) \la{s:AU}}
 We summarize now for comparison the results of \cite{AU} for the still tractable, more general extension of the \Seg-Tichy process provided by
 jump-diffusions with
 { affine premium and volatility}
\beq \la{aff}
\bc  c(x)= r x + c\\ \fr{\s^2(x)}2= \a_1 x + \a_0, \; \a_1, \a_0 \geq 0.
\ec  \eeq

Besides \OU type processes, \eqr{aff} includes
 another famous particular case, \CIR (CIR) type processes, obtained {when $\a_1 >0$}.

Introduce now a~{ combined ruin-survival} expected payoff at
time $t$
\begin{equation}
{V}(t,u)=\mathbb{E}_{ X_{0}=u  }\left[{w}\left(  X_{T }\right)
\,1_{\left\{  T \,<\,t\right\}  } + \,p(X_{t}%
)\,\,1_{\left\{  T \,\geq\,t\right\}  }\right]  \label{payoff}%
\end{equation}
where ${w}, p$
represent, respectively:

\begin{itemize}
\item A penalty ${w}(X_{T })$ at a~stopping time ${T }$, \; $\,
{w}:\;\mathbb{R} \mathbb{\rightarrow R}$

\item A reward for survival after $t$ years: $p(X_{t}%
),\,\,{\emph{p}}:
\;\mathbb{R} \mathbb{\rightarrow R}^{+}$.
\end{itemize}

Some particular cases of interest are the survival probability
for $t$ years, obtained with
\[
w(X_{T })=0,\;\;p(X_{t})=1_{\left\{  X_{t}\,\,\geq\,0\;\right\}
}\;\;
\]
and the ruin probability with deficit larger in absolute value than
$y$, obtained with
\[
w(X_{T })=1_{\left\{  X_{T }\,<\,-y\,\;\right\} },\;\;p(X_{t})=0\;\;
\]

Let \beq {V}_{q}(x)=\int_0^\I q e^{-q t} V(t,x) dt =E_x \left[{w}\left(  X_{T }\right)
\,1_{\left\{  T \,<\,\kil_q\right\}  } + \,p(X_{\kil_q}%
)\,\,1_{\left\{  T \,\geq\,\kil_q\right\}  }\right],\la{LCGS} \eeq
 denote a~ ``Laplace-Carson''/``Gerber Shiu'' discounted penalty/pay-off.

\begin{Pro} \cite[Lem. 1, Thm. 2]{AU} {(a)} Consider the \proc\  \ \eqr{aff}. Let $V_q(x)$ denote the corresponding \GS function \eqr{LCGS}, let $w_\lm(x)= \int_{x}^\I w(x -u) \lm(du)$ denote the expected payoff at ruin,
and let $g(x):=w_\lm(x)+q p(x), \H g(s)$ denote the combination of the two payoffs and its \LT; note that the particular cases $$\widehat{g}(s)=\fr \q s, \; \widehat{g}(s)=\l \ovl F(s)$$ correspond to the \sur and \rp, respectively \cite{AU}.

Then, the~Laplace transform of the derivative
 $$V_*(x)=\int_0^{\infty}e^{-sx}dV_q(x)={s}{\H V_q(s)- V_q(0)}$$
satisfies the ODE
\beq \la{AU*} &&  \left( \a_1 s + r\right) V_*(s)^{\prime}-    (\fr{\k(s)}{s}-\a_1)
  V_*(s)   = - h(s) {V}_{q}(0) -{\a_0}\,{V}_{q}^{\prime}(0) +\widehat{g}(s) \Lra \no\\&&
  V_*(s) I_{\q}(s)=\int_s^\I
   I_{\q}(y)\fr{h(y) V_{\q}(0) +{\a_0}\,{V}_{q}^{\prime}(0) -  \widehat{g} (y)}{r +\a_1 y}\; dy,\eeq
where $h(s)=\H {\ovl \lm}(s) +\fr \q{ s}$ (this corrects a~typo in \cite[eqn. (9)]{AU}), and where the integrating factor is obtained from \eqr{defI0} by replacing $c$ with $c-\a_1$ \cite[eqn. (11)]{AU}.
 Equivalently,
   \beq \la{AU} && r \left( s \H V_{\q}(s)\right)^{\prime}-    \fr{\k(s)}{s}
\,s \H V_{\q}(s)   = - (c +\a_0 s) {V}_{q}(0) -{\a_0}\,{V}_{q}^{\prime}(0)  +\widehat{g}(s)\Lra \no\\&&s \H V_{\q}(s) I_{\q}(s)=\int_s^\I I_{\q}(y)\fr{(c +\a_0 s) V_{\q}(0) +{\a_0}\,{V}_{q}^{\prime}(0) - \widehat{g} (y)}{r +\a_1 y} \; dy.\eeq

(b) If $\a_0=0=\a_1$ and $q>0$, the \sur probability satisfies
\beq \label{seg0q} && \sRui_{\q}(0) =\fr{\Tq \ovl I_{q-1}(0)}{\Tc \ovl I_{q}(0)} \la{s0} \\&&s \H{\sRui}_{\q}(s) I_{\q}(s) = \int_s^\I I_{\q}(y)(\Tc \sRui_{\q}(0)- \fr {\Tq}y)  dy=\Tc \sRui_{\q}(0) \ovl I_{\q}(s)-\Tq \ovl I_{q-1}(s)=\Tq\le(\fr{ \ovl I_{q-1}(0)}{ \ovl I_{q}(0)} \ovl I_{q}(s) - \ovl I_{q-1}(s)\ri) \no
 \eeq
\end{Pro}

\begin{proof}
{(b) The survival \pro \ follow from (a), by plugging $\widehat{g}(y)=\fr q y$. Indeed, the Equation~\eqr{AU} becomes for the \sur \proba}
\[s \H \sRui_{\q}(s) I_{\q}(s)=\int_s^\I I_{\q}(y)(\Tc \sRui_{\q}(0)  - \fr \Tq  y) \; dy = \Tc \sRui_{\q}(0) \ovl I_{\q}(s)-\Tq \ovl I_{\q-1}(s).\]

Letting $s \to 0$ in this equation yields $\sRui_{\q}(0) =\fr{\Tq \ovl I_{q-1}(0)}{\Tc \ovl I_{q}(0)}$.

As a~check, let us verify also Equation \eqr{AU} for the ruin \proba, by plugging $\widehat{g}(y)=\l \ovl F(y)$:
\bea &&s \H \Rui_{\q}(s) I_{\q}(s) = \int_s^\I I_{\q}(y)(\Tc \Rui_{\q}(0)-\l \ovl F(y)) dy=
\Tc \Rui_{\q}(0) \bar I_{\q}(y)-J(y),\\&& J(y)=\int_s^\I y^{\Tq }
 e^{- \Tc y}  j'(y)  dy, \;  j(y):=e^{\Tl \int_0^y \bar F(z) dz}.\eea

Integrating by parts, $J(y)= -I_{q}(s) +\Tc \ovl I_{\q}(s)-\Tq \ovl I_{q-1}(s)$. Finally,

\beq &&s \H \Rui_{\q}(s) I_{\q}(s) = \Tc (1-\sRui_{\q}(0)) \ovl I_{q}(s)-\Big(-I_{q}(s) +\Tc \ovl I_{\q}(s)-\Tq \ovl I_{q-1}(s)\Big)=\no\\&&I_{q}(s)+\Tq \ovl I_{q-1}(s)-\Tc \sRui_{\q}(0) \ovl I_{\q}(s)=I_{q}(s) -s \H{\sRui}_{\q}(s) I_{\q}(s).\eeq
\end{proof}

\ssec*{Segerdahl's Process via the  Laplace Transform\ Integrating Factor}
 We revisit now the particular case of Segerdahl's process with \expo \ claims of rate $\mu$ and $\a_0=\a_1=0$.
Using
$\ovl \lm(y)=\l F_C(y) dy=\fr \l{y + \mu}$ we find that
for Segerdahl's process the integrand in the exponent is
$$\fr{\k(s)}{rs}= \Tc -\Tl/(s+ \mu)-\Tq/s,$$ and the integrating factor \eqr{defI} may be taken as
$$I_{\q}(x)=x^{\Tq} e^{- \Tc x}(1+x/\mu)^{\Tl}.$$

The antiderivative $\bar I_\q(x)$ is not explicit, except
for:

\begin{enumerate}
\item $x=0$, when \ith
$$  \bar I_\q(0)= \mu^{\Tq +1 } U (\Tq +1,\Tq +\Tl +2,\Tc \mu ),$$
where \cite[{13.2.5}]{AS}
\footnote{Note that when $c=1$, this function
reduces
 to a~power:
$
U\left( a,a+1,z\right) =\frac{\int_{0}^{\infty}~t^{a-1}~e^{-z
t}~dt}{\Gamma(a)}=z^{-a}.
$}
$$U[a, a+c, z] = \fr 1{\Gamma[a]} \int_0^\I e^{-z t} t^{a - 1} (t + 1)^{c - 1} dt,  Re[z] > 0, Re[a] > 0.$$

\item for $q=0$, when \ith
$$I(x)=e^{- \Tc x}(1+x/\mu)^{\Tl}, \;  \bar I(x)=\int_x^\I I(y) dy=\frac{e^{\Tc \mu } (\Tc \mu)^{-\Tl } \Gamma (\Tl +1,\Tc (x+\mu ))}{\Tc}.
$$
\end{enumerate}

However, the~\LTs of the integrating factor $I_{\q}(x)$ and its primitive are explicit:
\beq \no  &&\H I_{\q}(s)=\int_0^\I e^{-(s+\Tc) x} x^{\Tq} (1+x/\mu)^{\Tl}=\G(\Tq +1)U(\Tq +1, \Tq +\Tl +2, \mu (\Tc+s),\\&&  \H {\ovl I}_{\q}(s)= \G(\Tq +1)\fr{U(\Tq +1, \Tq +\Tl +2, \mu \Tc)-U(\Tq +1, \Tq +\Tl +2, \mu (\Tc+s)}s. \la{LTs}\eeq

Finally, we may compute:

 \beq &&\sRui_{\q}(0) =\fr{\Tq \ovl I_{q-1}(0)}{\Tc \ovl I_{q}(0)}=\fr{\Tq U (\Tq,\Tq +\Tl +1,\Tc \mu )}{\Tc \mu U (\Tq +1,\Tq +\Tl +2,\Tc \mu )}\no \\&&\Rui_{\q}(0) =1-\sRui_{\q}(0)=1- \fr{\Tq U (\Tq,\Tq +\Tl +1,\Tc \mu )}{\Tc \mu U (\Tq +1,\Tq +\Tl +2,\Tc \mu )} \no \\&&=\fr{\Tc \mu U (\Tq +1,\Tq +\Tl +2,\Tc \mu )- {\Tq} U (\Tq,\Tq +\Tl +1,\Tc \mu )}{\Tc \mu U (\Tq +1,\Tq +\Tl +2,\Tc \mu )} \no \\&&=\left( \frac{\l}{c \mu }\right) \frac{ U\left( \T \q+1%
,\T \q + 1+\Tl,\mu   \Tc \right) }{U\left( \T \q+1,\T \q +
\Tl+2,\mu \Tc  \right) },\eeq
where we used the identity \cite[{13.4.18}]{AS}
 \beq \la{sPau} U[a-1, b, z]+(b-a) U[a, b, z]=z U[a, b+1,z], a>1,\eeq
 with $a =\Tq+1, b=\Tq + \Tl +1$.
 This checks the (corrected) Paulsen result \eqr{Pauruin} for $x=0$.

\beR We can now numerically answer Problem 1: (a) obtain the antiderivative $\bar I_\q(x)$ by numerical integration; (b) compute the \LT \ of the scale derivative by \eqr{wI}; c) Invert the \LT. \eeR

The example above raises the question:

\beQ Is it possible to compute explicitly the \LTs of the integrating factor $I_{\q}(x)$ and its primitive for affine processes
with phase-type jumps? \eeQ

\sec{Revisiting Segerdahl's Process via the Scale Derivative/Integrating Factor Approach \la{s:Segr}}

 Despite the new scale derivative/integrating factor approach, we were not able to produce further explicit results beyond \eqr{Pau}, due to the fact that neither the scale derivative, nor the integral of the integrating factor are explicit when $\q >0$ (this is in line with \cite{avram2010lie}). \eqr{Pau} remains thus for now an~outstanding, not well-understood exception.

 \beQ Are there other explicit \fp results for Segerdahl's process when $\q>0$? \eeQ

 In the next subsections, we show that via the scale derivative/integrating factor approach, we may rederive \wk \ results for $q=0$.

\ssec{Laplace Transforms of the Eventual Ruin and Survival Probabilities \la{s:Ak}}
For $\q=0$, both \LTs and their inverses are explicit, and~several classic results may be easily checked. The~scale derivative may be obtained using Proposition \ref{cor:wI} and
$ \Gamma (\Tl +1,v)=e^{-v} v^{\Tl}+ \lambda \Gamma (\Tl,v) $ with $v=\Tc (s+\mu )$.
 We find  \beq \la{scSeg} &&\widehat {\mathbf w}(s,a) =\fr{{e^{\Tc \mu } (\Tc \mu )^{-\lambda } \Gamma (\lambda +1,\Tc
   (s+\mu ))}}{e^{- \Tc s}(1+s/\mu)^{\Tl}}-1=1+ \lambda e^{v } (v )^{-\Tl } \Gamma (\Tl,v) -1=\Tl U(1,1+ \Tl,\Tc (s+\mu ))\no \\&& \Lra {\w}(x,a)=
\frac{\Tl}{{\Tc}}\left(1+ \frac{x}{\tilde{c}}\right)^{\Tl-1}
e^{-\mu x},\eeq
which checks \eqr{scder}.
Using again $\H {\w}(s)=\Tc \fr{ \ovl I(y)}{ I(y)} -1$ yields the ruin and \sur probabilities:
\bea &&s \H {\sRui}(s)  =\fr{ \int_s^\I \Tc \sRui(0)I(y) dy}{I(s)}=\sRui(0) ({\H \w}(s)+1)\\&&s \H \Rui(s)  = \fr{\int_s^\I (\Tc \Rui(0)-\fr{\Tl}{y+\mu})I(y) dy}{I(s)}=\Rui(0) ({\H \w}(s)+1)   - \H {\w}(s).
\eea

Letting $s \to 0$ yields
\beq \label{seg0} \no &&\Rui(0) =\fr{\H {\mathbf w}(0)}{\H {\mathbf w}(0)+1}=
 \frac{\Tl U\big(1, 1+\Tl,\mu \Tc\big)}
{\mu \Tc \; U\big(1, 2+\Tl,\mu \Tc\big)}=\fr{\Tl \Gamma (\Tl,\Tc \mu )} { \Gamma (\Tl +1,\Tc \mu )} \Eq\\&& \sRui(0) =\fr{\lim_{s \to 0} s \H {\sRui}(s)}{\H
{\mathbf w}(0)+1} = \frac{\sRui(\I)}
{1+ \Tl \; U\big(1, 1+\Tl,\mu \Tc\big)}=\frac{1}
{\mu \Tc \; U\big(1, 2+\Tl,\mu \Tc\big)}\eeq

For the survival probability, we finally find
\bea &&s \H {\sRui}(s) =
\sRui(0) (1+ \H {\mathbf w}(s))=
\fr{1 +\Tl  U\big(1, 1+\Tl,\mu (\Tc+s)\big)}{1 +\Tl  U\big(1, 1+\Tl,\mu \Tc\big)}=\fr{\Tc (\mu +s) U\big(1, 2+\Tl,\mu(\Tc+s)\big)}{\Tc \mu U\big(1, 2+\Tl,\mu \Tc\big)},\eea
which checks with the \LT \ of the Segerdahl result \eqr{segLT}.

\ssec{The Eventual Ruin and \sur \pros \la{s:Ak}}
These may also be obtained directly by integrating the explicit scale derivative ${\w}(x,a)=
\frac{\Tl}{{\Tc}}\left(1+ \frac{x}{\tilde{c}}\right)^{\Tl-1}
e^{-\mu x}$ \eqr{scSeg} 
Indeed,
\bea &&\int_u^\I {\mathbf w}(x) dx=\int_u^\I \frac{\Tl}{\tilde{c}}\left(1+ \frac{x}{\tilde{c}}\right)^{\Tl-1}
e^{-\mu x}dx= \Tl e^{\mu \Tc}\int_{1+ \frac{u}{\tilde{c}}}^\I y^{\Tl-1} e^{\mu \Tc y}dy\\&&=\Tl e^{\mu \Tc}
\fr {1}{(\mu \T c)^{\Tl}}
\int_{\mu(\T c+ u)}^\I t^{\Tl-1}
e^{- t}dt=\Tl e^{\mu \Tc}
(\mu \T c)^{-\Tl} \Gamma(\Tl,\mu(\T c+ u)),
\eea
 {where $\Gamma(\eta, x) = \int_x^{\infty} t^{\eta -1}e^{-t}dt$ is the
incomplete gamma function.}
The
\rp\ is \cite{Seg}, \cite[ex. 2.1]{Pau}:
\begin{eqnarray}\label{segint}
&&\Rui(x) = \Tl\frac{\exp(\mu
  \Tc )(\mu \Tc)^{-\Tl }\Gamma\big(\Tl,\mu (\Tc + x)\big)}
{1+\Tl\exp(\mu
  \Tc )(\mu \Tc)^{-\Tl }\Gamma(\Tl,\mu \Tc)}= \Tl\frac{e^{-\mu
  x }(1+ x/\Tc)^{\Tl}\; U\big(1, 1+\Tl,\mu (\Tc +x)\big)}
{1+ \Tl U\big(1, 1+\Tl,\mu \Tc \big)} \no\\&&=\fr{\Tl}{\mu \Tc}  \frac{e^{-\mu
  x }(1+ x/\Tc)^{\Tl  }\; U\big(1, 1+\Tl,\mu (\Tc +x)\big)}
{ U\big(1, 2+\Tl,\mu \Tc \big)}=\fr{\Tl \Gamma (\Tl, \mu(\Tc+x) )} { \Gamma (\Tl +1,\Tc \mu )},
\end{eqnarray}
 where we used
\beq \la{TriG} U\big(1, 1+\Tl,v)= e^v v^{-\Tl } \Gamma (\Tl,v)\eeq
and \beq \la{sU} 1+ \Tl U\big(1, 1+\Tl,v \big)= v U\big(1, 2+\Tl,v \big),\eeq which holds by integration by parts.

A simpler formula holds for the rate of ruin $\rui(x)$ and its \LT
\beq \la{segLT}&& \rui(x)=-\Rui'(x)=\fr{{\mathbf w}(x)}{1+ \int_0^\I {\mathbf w}(x) dx}=
\fr{\Tl}{ \Gamma (\Tl +1,\Tc \mu )} \mu (\mu(\Tc+x) )^{\Tl-1}e^{-\mu(\Tc+x)} =e^{-\mu \Tc}\g_{\Tl,\mu}(x + \Tc) \Eq \no \\&&
\H \rui(s)=\sRui(0) \H{\mathbf w}(s)=\bc
 \fr{\Tl U(1,1+ \Tl,\Tc (s+\mu ))}{\Tc \mu U(1,2+ \Tl,\Tc \mu )}, &c> 0\\ (1+s/\mu )^{-\Tl }, &c= 0\ec,\eeq where $\g$ denotes a~(shifted) Gamma density.
Of course, the~case $c>0$ simplifies to a~Gamma density when moving the origin to the “absolute ruin'' point $-\Tc=-\fr c r,$, i.e., by putting $y=x+ \Tc, Y_t=X_t + \Tc$,
where the process $Y_t$ has drift rate $r Y_t$.

\beQ Find a~relation between the ruin derivative $\rui_\q(x)=-\Rui_q'(x)$ and the scale derivative ${\mathbf w_\q}(x)$ when $q>0$. \eeQ

\section{Asmussen's Embedding Approach for Solving Kolmogorov's Integro-Differential Equation
with Phase-Type Jumps \la{s:CLSa}} One of the most convenient approaches to get rid of the integral term in \eqr{reneq} is a~probabilistic transformation which gets rid of the jumps as in  \cite{asmussen1995stationary}, when the downward
phase-type jumps have a
survival \fun $$\bar{F}_C(x)=\int_x^{\infty} f_C(u) du= {\vec \b}e^{B x}
 \bff 1,$$ where ${ B}$ is a~$n\times n$ stochastic generating matrix
(nonnegative off-diagonal elements and nonpositive row sums),
$\vec \b=(\b_1,\ldots,\b_n)$ is a~row probability vector (with nonnegative elements and
$\sum_{j=1}^n\b_j=1$), and
$\bff 1=(1,1,...,1)$ is a~column probability vector.

The density is ${f_C(x)}=\vec \b e^{-B x} \bff b$, where ${\bff b}=(-B) \bff 1$
is a~column vector, and~the Laplace transform is 
$$\hat{b}(s)={\vec \b}(s I -\bff B)^{-1} \bff b.$$

Asmussen's approach \cite{asmussen1995stationary,asmussen2002erlangian} replaces the negative jumps by segments of slope $-1$, embedding the original \sn \lev process into a~continuous Markov modulated \lev process. For the new process we have auxiliary unknowns {$A_i(x)$} representing ruin or survival probabilities (or, more generally, Gerber-Shiu functions) when starting at $x$ conditioned {on} a~phase $i$ with drift downwards (i.e., in one of the “auxiliary stages of artificial time'' introduced by changing the jumps to segments of slope $-1$).
Let ${\bf A}$ denote the column vector
with components $A_1,\ldots,A_n$. The~ Kolmogorov integro-differential equation turns then into a~system of ODE's, due to the continuity of the embedding process.

\begin{equation}\label{Line}
\left(
\begin{aligned}
 \Rui_\q'(x)\\
{\bf A}'(x)\\
\end{aligned}\right)=
\left(
\begin{matrix}
\frac{\lambda+q}{c(x)} & -\frac{\lambda}{c(x)}{\vec \b}\\
 {\bf b} &  {\bf B}\\
\end{matrix}\right)
\left(
\begin{aligned}
 \Rui_\q(x)\\
{\bf A}(x)\\
\end{aligned}\right),\ {x\geq 0.}
\end{equation}

 For the \rp\ with exponential jumps {of rate $\mu$} for example, there is only one downward phase, and~the system is:
\begin{equation} \label{expoLine} \left(\begin{array}{c}
   \ {\Rui}_\q'(x)\\   A'(x)
\end{array} \right) = \left(\begin{array}{cc}
\frac{\lambda +q}{c(x)} &- \frac{\lambda }{c(x)} \\  \mu & -\mu
\end{array} \right)  \left(\begin{array}{c}
  \Rui_\q(x)\\   A(x)
\end{array} \right)\ {x\geq 0}.\end{equation}
 {For survival probabilities,
one only needs to modify the boundary conditions---see the following~section.}

\ssec{Exit Problems for the
\Seg-Tichy process, with $\q=0$}
Asmussen's approach is particular convenient for solving exit problems for the
\Seg-Tichy process.

\beXa {\bf The eventual ruin probability}.
 {{\bf When $q=0,$} the system for the ruin probabilities {with $x\geq0$} is:}
\begin{equation}
\left\{\begin{aligned} {\Rui}'(x)&= \frac{\lambda }{c(x)} \
(\Rui(x) - A(x)), \quad   \; \; \Rui(\infty)&=
A(\infty)= 0\\
 A'(x)&= \mu  \ (\Rui(x)- A(x)),  \
 \quad \quad  \; \;  &A(0)= 1
\end{aligned}\right.
\end{equation}

This may be solved by subtracting the equations. Putting \bea K(x)=e^{-\mu
x + \int_0^x \frac{\lambda} {c(v)} d v}, \eea we find:
\begin{equation} \label{ZSeg}
\bc \Rui(x)-A(x)&=(\Rui(0)-A(0))
K(x),   \\
A(x)&=  \mu (A(0)-\Rui(0)) \int_{x}^\I K(v) d v,
\ec,
\end{equation}
whenever $K(v)$ is integrable at $\infty$.

The boundary
condition $
A(0)=1$ implies that $1-\Rui(0)=\fr{1}{\mu \int_0^{\infty} K(v) d
v}$ and
\bea &&
\quad A(x)= \mu (1-\Rui(0)) \int_x^{\infty} K(v)
d v=\fr{\int_x^{\infty} K(v) d v}{\int_0^{\infty} K(v)
d v},\\&&\Rui(x)-A(x)={-\fr{K(x)}{\mu \int_0^{\infty} K(v) d v}}.\eea

Finally,
\bea &&\Rui(x)=A(x)+\left(\Rui(x)-A(x)\right)=\fr{ \mu
\int_x^\I K(v) d v - K(x)}{\mu \int_0^{\infty} K(v) d v},
\eea
 and {for the survival probability $\sRui$,}
\beq \la{seg1} && \sRui(x)=\fr{\mu
\int_0^{x} K(v) d v + K(x)}{\mu \int_0^{\infty} K(v) d v}:= \sRui(0) {\mathbf W}(x)=\frac{{\mathbf W}(x)}{{\mathbf
W}(\infty )}, \\&& {\mathbf W}(x)=\mu  \int_0^x
K(v) d v + K(x), \no\eeq
where $\sRui(0) =\frac{1}{{\mathbf
W}(\infty )}$ by plugging ${\mathbf
W}(0 )=1$ in the first and last terms in \eqr{seg1}.

We may also rewrite \eqr{seg1} as:
\begin{eqnarray}\la{scder}  &&\sRui(x)=\frac{1+\int_0^x \w(v) d v}{1+\int_0^\infty
\w(v) d v} \Leftrightarrow  \Rui(x)=\frac{\int_x^\infty  \w(v) d v}{1+\int_0^\infty  \w(v)
d v},  \w(x):={\mathbf W}'(x)=\fr{\lambda  K(x)}{c(x)}\end{eqnarray}

Note that $\w(x) >0$ implies that the scale function ${\mathbf W}(x)$ is
nondecreasing.
\eeXa



\beXa For the two sided exit problem on $[a,b]$, 
a similar derivation yields the
scale function
$${\mathbf W}(x,a)=\mu  \int_a^x
\fr{K(v)}{K(a)} d v + \fr{K(x)}{K(a)}=1+\fr{1}{K(a)} \int_a^x \w(y) dy,$$
with scale derivative derivative
   $ \w(x,a)=\fr{1}{K(a)}\w(x)$,
where
$\w(x)$ given by \eqr{scder} does not depend on $a$.

Indeed, the~ analog of \eqref{ZSeg} is:
\begin{equation}
\left\{\begin{aligned} \label{ZSegc} \sRui^{b}(x,a)-A^{b}(x)&=\sRui^{b}(a,a) \fr{K(x)}{K(a)},   \nonumber\\
A^{b}(x)&= \mu \sRui^{b}(a,a)  \int_a^x \fr{K(v)}{K(a)} d v,
\end{aligned}\right.\end{equation}
implying {by the fact that $\sRui^{b}(b,a)=1$ that}
\beq &&\sRui^{b}(x,a)=\sRui^{b}(a,a) \le(\fr{K(x)}{K(a)}+\mu \int_a^x \fr{K(v)}{K(a)} d v\ri)=\fr{{\mathbf W}(x,a)}
{{\mathbf W}(b,a)}=\fr{1+ \fr{1}{K(a)}\int_a^x\w(u) d u} {1+\fr{1}{K(a)}\int_a^b\w(u)
d u}  \Eq \no \\ && \Rui^{b}(x,a)= \fr{\int_x^b\w(u) d u} {K(a)+\int_a^b\w(u)
d u} \Eq \la{pdscale} \\&& \rui^{b}(x,a):=- (\Rui^{b})'(x,a)=  \fr{ \w(x)} {K(a)+\int_a^b\w(u)
d u}=\w(x,a)\fr{ \sRui(a,a)} {\sRui(b,a)}. \no\eeq 

\eeXa


\beR The definition adopted in this section for the scale function ${\mathbf W}(x,a)$ uses the normalization ${\mathbf{W}(a,a)}=1$, which~is only appropriate in the absence of Brownian motion. 
\eeR

\beQ Extend the \eqs \; for the \sur and \rp\ of the Segerdahl-Tichy process in terms of the scale derivative
$\w_\q$, when $q>0$.
Essentially, this requires obtaining
$$T_q(x)=E_x\Big[ e^{-q [T_{a,-} \min  T_{b,+}]} \Big ]$$
\eeQ

\sec{Further Details on the Identities Used in the Proof of Theorem \ref{t:W} \la{s:a}}

{We recall first some continuity and differentiation relations needed here \cite{AS}}
\beP Using the notation $M = M(a, b, z), M(a+) = M(a + 1, b, z), M(+,+) = M(a + 1, b+1, z)$, and~so on,
 the Kummer and
Tricomi functions satisfy the following identities:
\begin{equation*}
  b M+ (a- b) M\left(
b+ \right)=a M\left(a+ \right) \tag{13.4.3}\end{equation*}
\begin{equation*} b\big(M(a+)- M \big) =z M\left(+,+ \right) \tag{13.4.4} 
\end{equation*}
\begin{equation*}(b - a) U +
 z  U( b+2 ) = (z + b) U(b+1) \tag{13.4.16} \end{equation*} \begin{equation*}
 U+ a~U\left(+,+ \right) = U\left(
b+ \right) \tag{13.4.17}\end{equation*}\begin{equation*}U+(b-a-1) U(a+1)=z U(+,+) \tag{13.4.18}\end{equation*}
(see corresponding equations in \cite{AS}).
\beq \la{der}
U'=-a U\le(+,+\ri), \quad M'=\fr{a}{b}
M\le(+,+\ri).\eeq

\eeP

\beP

The functions $K_i(\Tq,\Tl,z)$ defined by \eqr{Ki} satisfy
the identities
\beq K_1'(\Tq,n, z)&=&(\Tq+\Tl)e^{-z} z^{\Tq + \Tl-1} \
 M \le(\Tq, \Tq + \Tl, z\ri)=(\Tq + \Tl) K_1(\Tq-1,\Tl,z) \la{Kids1}\\
 K_2'(\Tq,n, z)&=&- e^{-z}
z^{\Tq + \Tl-1} \
 U \le(\Tq, \Tq + \Tl, z\ri)=- K_2(\Tq-1,n, z) \la{Kids2}\eeq

 \beq K_2(\Tq,n, z)= \int_z^{\infty}
({y-z})^{\T \q} \, (y )^{ n-\Tq-1} \,
 e^{- y}
 {d y} 
  \label{UIdent}
\end{eqnarray}

\eeP
\prf 
For the first identity, note, using (\cite{AS} \cite{AS}, {13.4.3, 13.4.4}), that
\bea \fr{e^{z} }{z ^{\Tq + \Tl-1}} K_1'(z)&=&  ({\Tq+\Tl}-{z}) M\left( \Tq+1,\Tq
+1+ \Tl,z\right) + z \fr{\Tq + 1}{\Tq + \Tl + 1} M\left( \Tq+2,\Tq
+2+ \Tl,z\right)
\\&=&  ({\Tq+\Tl})M\left( \Tq+1,\Tq +1+ \Tl,z\right)\\&+&  \fr{z}{\Tq + \Tl + 1} \le((\Tq + 1) M\left( \Tq+2,\Tq +2+ \Tl,z\right)-({\Tq + \Tl +1})
M\left( \Tq+1,\Tq +1+ \Tl,z\right) \right)\\&=& ({\Tq+\Tl})M\left( \Tq+1,\Tq +1+ \Tl,z\right)-  \fr{z}{\Tq + \Tl + 1} \Tl
M\left( \Tq+1,\Tq +2+ \Tl,z\right) \\&=&  ({\Tq+\Tl})M\left( \Tq+1,\Tq +1+ \Tl,z\right)-   \Tl\le(M\left( \Tq+1,\Tq +1+ \Tl,z\right)-M\left( \Tq,\Tq +1+ \Tl,z\right)\ri)
\\&=&  \Tq M\left( \Tq+1,\Tq +1+ \Tl,z\right) +   \Tl
M\left( \Tq,\Tq +1+ \Tl,z\right).\eea

The second formula may be derived similarly using 13.4.17, or by considering the
function
  \bea _z {\T U}(\Tq+1, \Tq+1+\Tl,\mu):=
  \Gamma(q+1) K_2(z)  =\int_z^{\infty} ({s-z})^{\Tq} \,
({s} )^{  \Tl-1} \,
 e^{-  \mu s} {d s} \eea
 appearing in the numerator of the last form of \eqref{UIdent}.
 An integration by parts yields
 \bea &&_z {\T U}'(\Tq+1, \Tq+1+\Tl,1)=\int_z^{\infty} ({s-z})^{\Tq} \,
 \fr{d}{d z}[
({s} )^{  \Tl-1} \,
 e^{-   s}] {d s}\no \\&&=  (\Tl -1)
_z {\T U}(\Tq+1,\Tq+ \Tl,1) - _z {\T U}(\Tq+1, \Tq+\Tl+1,1), \; \Lra
\\&&
 K_2'(\Tq+1, \Tl,z)=  e^{-z}
z^{\Tq + \Tl-1}\le((\Tl -1)
 U(\Tq+1, \Tq +\Tl,z) -
 U (\Tq+1, \Tq +\Tl+1,z)\ri) \eea
 and the result follows by (\cite{AS} \cite{AS},
 {13.4.18.})\footnote{See also \cite[p. 640]{BS}, where however the first formula has a~typo.}

The third formula is obtained by the substitution $y=z(t+1)$.

\sec{Conclusions and Future Work} \la{s:con}
Two promising fundamental functions have been proposed for working with generalizations of Segerdahl's process: (a) the scale derivative $\w$ \cite{CPRY} and (b) the integrating factor $I$ \cite{AU}, and~they are shown to be related via Thm. 1.

Segerdahl's process per se is worthy of further investigation.
A priori, many risk problems (with absorbtion/reflection at a~barrier $b$ or with double reflection, etc.) might be solved by combinations of the hypergeometric functions $U$ and $M$.

However, this approach leads to an~impasse for more complicated jump structures, which~ will lead to more complicated hypergeometric functions.
In that case, we would prefer answers expressed in terms of the fundamental functions $\w$ or $I$.

We conclude by mentioning two promising numeric approaches, not discussed here. One due to \cite{JJ} bypasses the need to deal with high-order hypergeometric solutions by employing complex contour integral representations. The~second one uses Laguerre-Erlang expansions---see \cite {abate1996laguerre,ALR}. 
 Further effort of comparing their results
with those of the methods discussed above seems worthwhile.

\small
\bibliographystyle{amsalpha}

\bibliography{Pare37}
\end{document}